\documentclass[12pt]{amsart}
\usepackage{amsmath}
\usepackage{amssymb}
\usepackage{latexsym}
\usepackage{amscd}
\usepackage{citesort}
\usepackage{graphicx} %We can use any other package if necessary
\usepackage{amsthm}
\usepackage{mathrsfs}
\usepackage{xypic}
\usepackage{bm}

\newdimen\AAdi%
\newbox\AAbo%
\def\AAk#1#2{\s_etbox\AAbo=\hbox{#2}\AAdi=\wd\AAbo\kern#1\AAdi{}}%
\def\AAr#1#2#3{\s_etbox\AAbo=\hbox{#2}\AAdi=\ht\AAbo\raise#1\AAdi\hbox{#3}}%
\font\tenmsb=msbm10 at 12pt \font\sevenmsb=msbm7 at 8pt
\font\fivemsb=msbm5 at 6pt
\newfam\msbfam
\textfont\msbfam=\tenmsb \scriptfont\msbfam=\sevenmsb
\scriptscriptfont\msbfam=\fivemsb
\def\Bbb#1{{\tenmsb\fam\msbfam#1}}
\newtheorem{thm}{Theorem}[section]
\newtheorem{lem}{Lemma}[section]

\newtheorem{rem}{Remark}[section]
\newtheorem{pro}{Proposition}[section]

\newcommand{\ba}{\begin{array}}
\newcommand{\ea}{\end{array}}

\newcommand{\Section}[2]{\setcounter{equation}{0}
\allowdisplaybreaks
\section[#1]{#2}}

\def\n{\nabla}

\def\ir#1{\mathbb R^{#1}}

\def\f#1#2{\frac{#1}{#2}}

\def\grs#1#2{\bold G_{#1,#2}}

\def\pd#1#2{\frac {\partial #1}{\partial #2}}

\def\td{\tilde}

\def\a{\alpha}
\def\be{\beta}

\def\p#1{\partial #1}

\def\de{\delta}
\def\De{\Delta}
\def\e{\eta}
\def\ep{\varepsilon}
\def\G{\Gamma}
\def\g{\gamma}

\def\la{\lambda}
\def\La{\Lambda}
\def\om{\omega}
\def\Om{\Omega}
\def\th{\theta}

\def\si{\sigma}

\def\w{\wedge}

\def\Hess{\mbox{Hess}}
\def\R{\Bbb{R}}
\def\tr{\mbox{tr}}
\def\U{\Bbb{U}}
\def\lan{\langle}
\def\ran{\rangle}
\def\ra{\rightarrow}

\begin{document}
\title
[convex functions  and Lawson-Osserman problem ] {convex functions
on Grassmannian manifolds and Lawson-Osserman problem}

\author
[Y.L. Xin and Ling Yang]{Y. L. Xin and Ling Yang}
\address
{Institute of Mathematics, Fudan University, Shanghai 200433, China
and Key Laboratory of Mathematics for Nonlinear Sciences (Fudan
University), Ministry of Education} \email{ylxin@fudan.edu.cn}
\thanks{The research was partially supported by
NSFC and SFECC}
\begin{abstract}
We derive estimates of the Hessian of two smooth functions defined
on Grassmannian manifold. Based on it, we can derive curvature
estimates for minimal submanifolds in Euclidean space via Gauss map
as \cite{x-y}. In this way, the result for Bernstein type theorem
done by Jost and the first author could be improved.
\end{abstract}

\renewcommand{\subjclassname}{%
  \textup{2000} Mathematics Subject Classification}
\date{}
\maketitle

\Section{Introduction}{Introduction}
\medskip

The celebrated theorem of Bernstein \cite{b} says that the only
entire minimal graphs in Euclidean 3-space are planes. Its higher
dimensional generalization was finally proved by J. Simons
\cite{si}, which says that an entire minimal graph has to be planar
for dimension$\leq 7$, while Bombieri- De Giorgi-Giusti \cite{b-d-g}
shortly afterwards produced a counterexample to such an assertion in
dimension 8 and higher.

Schoen-Simon-Yau \cite{s-s-y} gave us a direct proof for Bernstein
type theorems for $n\le 5$ dimensional minimal graphs with the aid
of  curvature estimates for stable minimal hypersurfaces.

There is a weak version of Bernstein type theorem in arbitrary
dimension. It was J. Moser \cite{mo} who proved that the entire
solution $f$ to the minimal surface equation is affine linear,
provided $|\n f|$ is uniformly bounded. Afterward Ecker-Huisken
\cite{e-h} obtained  curvature estimates by a geometric approach, as
a corollary Moser's result had been improved  for the controlled
growth of $|\n f|$.

Moser's theorem had been generalized to certain higher codimensional
cases by Chern-Osserman  \cite{c-o} for dimension $2$ and Babosa,
Fischer-Colbrie for dimension $3$ \cite{ba}, \cite{f-c}.  But the
counterexample constructed by Lawson-Osserman \cite{l-o} prevents us
going further.  They also raised in the same paper a question for
finding the "best" constant possible in the theorem.

In contrast, the first author with J. Jost \cite{j-x} proved the
following Bernstein type theorem without the restriction of
dimension and codimension, which is an improvement of the work done
by Hildebrandt-Jost-Widman \cite{h-j-w}.

\begin{thm}\label{b}

Let $z^\a=f^\a(x^1,\cdots,x^n),\ \a=1,\cdots,m$, be smooth functions
defined everywhere in $\R^n$. Suppose their graph $M=(x,f(x))$ is a
submanifold with parallel mean curvature in $\R^{n+m}$. Suppose that
there exists a number $\be_0$ with
\begin{equation}
\be_0<\left\{\begin{array}{cl}
2 & \mbox{when }m\geq 2,\\
\infty & \mbox{when }m=1;
\end{array}\right.\label{be1}
\end{equation}
such that
\begin{equation}
\De_f=\left[\text{det}\left(\de_{ij}
     +\sum_\a\pd{f^\a}{x^i}\pd{f^\a}{x^j}\right)\right]^{\f{1}{2}}\leq
     \be_0.\label{be2}
\end{equation}
Then $f^1,\cdots,f^m$ has to be affine linear representing an affine $n$-plane.

\end{thm}

The key point of the proof is to find a
geodesic convex set
\begin{equation*}
B_{JX}(P_0)=\big\{P\in \grs{n}{m}:\mbox{ sum of any two Jordan angles between }P\mbox{ and }P_0<\f{\pi}{2}\big\}
\end{equation*} in a geodesic polar coordinate of the Grassmannian manifold, where $P_0$ denotes a fixed
$n-$plane. It is larger than the largest geodesic convex ball of
radius $\f{\sqrt{2}}{4}\pi$ in $\grs{n}{m}$. The geometric meaning
of the condition of the above result is that the image under the
Gauss map of $M$ lies in a closed subset $S\subset B_{JX}(P_0)$.

Recently, the authors \cite{x-y} studied complete minimal
submanifolds whose Gauss image lies in an open geodesic ball of
radius $\f{\sqrt{2}}{4}\pi$. They carried out the Schoen-Simon-Yau
type curvature estimates and the Ecker-Huisken type curvature
estimates, and on the basis, the corresponding Bernstein type
theorems with dimension limitation or growth assumption could be
derived.

It is natural to study the situation when $\be_0$ in the condition
(\ref{be1}) and (\ref{be2}) of the  Theorem \ref{b} approach to $2$.
The present paper will devote to this problem. We shall follow the
main idea of our previous paper \cite{x-y}. But, we view now the
Grassmannian manifolds as submanifolds in Euclidean space via
Pl\"ucker imbedding. The auxiliary  functions are constructed from
this viewpoint.  As shown before, $B_{JX}(P_0)$ is defined in a
coordinate neighborhood $\U$ of the Grassmannian $\grs{n}{m}$. We
introduce two functions $v$ and $u$ in $\U$. Via the Gauss map we
can obtain useful functions on our minimal $n-$submanifold $M$ in
$\R^{m+n}$ with $m\ge 2$. Then, we can carry out the
Schoen-Simon-Yau type curvature estimates and the Ecker-Huisken type
curvature estimates, which enable us to get the corresponding
Bernstein type theorems and other geometrical conclusions.

In Section 2, we give some facts of a Grassmannian manifold
$\grs{n}{m}$, which can be isometric imbedding into a Euclidean
space. There is the height function for a submanifold in Euclidean
space. Such a height function is called $w-$function on
$\grs{n}{m}$. Then we have an open domain of $\U\subset \grs{n}{m}$,
where the $w-$function is positive. Every point in $\U$ has a
one-to-one correspondence to an $n\times m$ matrix. We describe
canonical metric and the corresponding connection on $\U$ with
respect to the coordinate. On the basis, the Hessian of an arbitrary
smooth function could be calculated.

In Section 3 we define $v=\f{1}{w}$ on $\U$. We also define another
function $u$ on $\U$. In the section, we shall show $v$ and $u$ are
convex on $B_{JX}(P_0)$ and give   estimates of the Hessian of them.
The estimates are quite delicate. We use the radial compensation
technique to accurate the estimates.

In Section 4, we define four auxiliary functions, $\td{h}_1$,
$\td{h}_2$, $\td{h}_3$ and $\td{h}_4$. They are defined on the
minimal submanifolds of $\R^{n+m}$ whose Gauss image is confined,
and they are expressed in term of $v$ and $u$. We also estimate the
Laplacian of them, which is useful for the next sections.

Later in Section 5, not only we give the Schoen-Simon-Yau type
curvature estimates with the aid of $\td{h}_1$ and $\td{h}_3$, but
also we obtain the Ecker-Huisken type curvature estimates with the
aid of $\td{h}_2$ and $\td{h}_4$. Our method is completely similar
to the previous paper \cite{x-y}, so we only describe the outline of
process. From the estimates several geometrical conclusions follow,
including the following Bernstein type theorems.

\begin{thm}\label{b1}

Let
 $M=(x,f(x))$ be an $n$-dimensional entire minimal graph given by $m$ functions
$f^\a(x^1,\cdots,x^n)$ with $m\geq 2, n\leq 4$. If
$$\De_f=\left[\text{det}\left(\de_{ij}+\sum_\a\pd{f^\a}{x^i}\pd{f^\a}{x^j}\right)\right]^{\f{1}{2}}<2,$$
then $f^\a$ has to
be affine linear functions representing an affine $n$-plane.

\end{thm}

\begin{thm}\label{b2}
Let $M=(x,f(x))$ be an $n$-dimensional entire minimal graph given by
$m$ functions $f^\a(x^1,\cdots,x^n)$ with $m\geq 2$. If
$$\De_f=\left[\text{det}\left(\de_{ij}+\sum_\a\pd{f^\a}{x^i}\pd{f^\a}{x^j}\right)\right]^{\f{1}{2}}<2,$$
and
\begin{equation}
\left(2-\De_f\right)^{-1}=o(R^{\f{4}{3}}),
\end{equation}
where $R^2=|x|^2+|f|^2$.
Then $f^\a$ has to be affine linear functions and hence $M$ has to be
an affine linear subspace.
\end{thm}

Those are what shall be done in Section 6. It is worthy to note that
Theorems \ref{b1}-\ref{b2} still hold true when $M$ is a submanifold
with parallel mean curvature. Dong generalized Chern's result
\cite{ch} \cite{d} to higher codimension, which states that a
graphic submanifold $M=(x,f(x))$ with parallel mean curvature has to
be minimal if the slope of $f$ is uniformly bounded. Hence our
results improve Theorem \ref{b}.

It is natural to ask what is the relations between the results here
and that of the previous paper \cite{x-y}. Since the $v-$function
varies in $\left(
\sec^p\left(\f{\pi}{2\sqrt{2p}}\right),\sec\left(\f{\sqrt{2}}{4}\pi\right)\right)$
on the open geodesic ball of radius $\f{\sqrt{2}}{4}$ in
$\grs{n}{m}$, where $p=\min(n, m)$, the results of the present
article do not generalize those in the previous one. Both results
are complementary.

\bigskip\bigskip

\Section{Preliminaries on the Grassmannian manifold
$\grs{n}{M}$}{Preliminaries on the Grassmannian manifold
$\grs{n}{m}$}
\medskip

Let $\ir{n+m}$ be an $n+m$-dimensional Euclidean space. All oriented
$n$-subspaces constitute the Grassmannian manifolds $\grs{n}{m}$, which
is an irreducible symmetric space of compact type.

Fix $P_0\in \grs{n}{m}$ in the sequel, which is spanned by a unit
$n-$vector $\ep_1\w\cdots\w\ep_n$. For any $P\in\grs{n}{m}$, spanned
by a $n-$vector $e_1\w\cdots\w e_n$, we define an important function
on $\grs{n}{m}$
$$w\mathop{=}\limits^{def.}\left<P, P_0\right>
=\left<e_1\wedge\cdots\wedge e_n, \ep_1\wedge\cdots\wedge
\ep_n\right>=\det W,$$
where $W=(\left<e_i, \ep_j\right>).$   It is
well known that
$$W^TW=O^T\La O,$$
where $O$ is an orthogonal matrix and
$$\La=
\begin{pmatrix} \mu_1^2&&0\cr
          &\ddots&\cr
      0&&\mu_p^2
  \end{pmatrix},\qquad p=\min(m,n),$$
where each $0\le\mu_i^2\le 1.$
The Jordan angles between $P$ and $P_0$ are defined by
$$\th_i=\arccos(\mu_i).$$

Denote
\begin{equation*}
\U=\{P\in \grs{n}{m}:w(P)>0\},
\end{equation*}
let $\{\ep_{n+\a}\}$ be $m$-vectors such that
$\{\ep_i,\ep_{n+\a}\}$ form an orthornormal basis of $\ir{m+n}$.
Then we can span arbitrary $P\in \U$ by $n$ vectors $f_i$:
$$f_i=\ep_i+z_{i\a}\ep_{n+\a},$$
where $Z=(z_{i\a})$ are the local coordinate of $P$ in $\Bbb{U}$.
Here and in the sequel we use the summation convention and agree the
range of indices:
$$1\leq i,j,k,l\leq n;\qquad 1\leq \a,\be,\g,\de\leq m.$$

The canonical metric on $\grs{n}{m}$ in the local coordinate can be
described as (see \cite{x2} Ch. VII)
\begin{equation}\label{metric}
g=\tr\big((I_n+ZZ^T)^{-1}dZ(I_m+Z^TZ)^{-1}dZ^T\big).
\end{equation}

Let $P\in \U$ determined by an $n\times m$ matrix
$Z_0=\big(\la_\a\de_{i\a}\big)$, where $\la_\a=\tan\th_\a$ and
$\th_1,\cdots,\th_m$ be the Jordan angles between $P$ and $P_0$.
(Here and in the sequel we assume $n\geq m$ without loss of
generality; for it is similar for $n<m$.) Let $X, Y, W$ denote
arbitrary $n\times m$ matrices. Then (\ref{metric}) tells us
\begin{eqnarray}\aligned
\lan X,Y\ran_P&=\tr\big((I_n+Z_0Z_0^T)^{-1}X(I_m+Z_0^TZ_0)^{-1}Y^T\big)\\
&=\sum_{i,\a}(1+\la_i^2)^{-1}(1+\la_\a^2)^{-1}X_{i\a}Y_{i\a}.\endaligned\label{metric1}
\end{eqnarray}
(Note that if $m+1\leq i\leq n$, $\la_i=0$.)
Furthermore, from
\begin{eqnarray*}\aligned
&\big(I_n+(Z_0+t W)(Z_0+t W)^T\big)^{-1}X\big(I_m+(Z_0+t W)^T(Z_0+t W)\big)^{-1}Y^T\\
&=\big(I_n+Z_0Z_0^T+t(WZ_0^T+Z_0W^T)+O(t^2)\big)^{-1}X\\
&\qquad\big(I_m+Z_0^TZ_0+t(W^TZ_0+Z_0^TW)+O(t^2)\big)^{-1}Y^T\\
&=\big(I_n+t(I_n+Z_0Z_0^T)^{-1}(WZ_0^T+Z_0W^T)+O(t^2)\big)^{-1}(I_n+Z_0Z_0^T)^{-1}X\\
&\qquad\big(I_m+t(I_m+Z_0^TZ_0)^{-1}(W^TZ_0+Z_0^TW)+O(t^2)\big)^{-1}(I_m+Z_0^TZ_0)^{-1}Y^T\\
&=\big(I_n-t(I_n+Z_0Z_0^T)^{-1}(WZ_0^T+Z_0W^T)+O(t^2)\big)(I_n+Z_0Z_0^T)^{-1}X\\
&\qquad\big(I_m-t(I_m+Z_0^TZ_0)^{-1}(W^TZ_0+Z_0^TW)+O(t^2)\big)(I_m+Z_0^TZ_0)^{-1}Y^T\\
&=(I_n+Z_0Z_0^T)^{-1}X(I_m+Z_0^TZ_0)^{-1}Y^T\\
&\quad-t\big[(I_n+Z_0Z_0^T)^{-1}(WZ_0^T+Z_0W^T)(I_n+Z_0Z_0^T)^{-1}X(I_m+Z_0^TZ_0)^{-1}Y^T\\
&\qquad+(I_n+Z_0Z_0^T)^{-1}X(I_m+Z_0^TZ_0)^{-1}(W^TZ_0+Z_0^T
W)(I_m+Z_0^TZ_0)^{-1}Y^T\big]+O(t^2),
\endaligned
\end{eqnarray*}
we have
\begin{eqnarray}\aligned
W\lan
X,Y\ran_{P}=&-\tr\big[(I_n+Z_0Z_0^T)^{-1}(WZ_0^T+Z_0W^T)\\
&\qquad\qquad (I_n+Z_0Z_0^T)^{-1}X(I_m+Z_0^TZ_0)^{-1}Y^T\\
&\qquad+(I_n+Z_0Z_0^T)^{-1}X(I_m+Z_0^TZ_0)^{-1}\\
&\qquad\qquad
(W^TZ_0+Z_0^TW)(I_m+Z_0^TZ_0)^{-1}Y^T\big].\endaligned\label{metric2}
\end{eqnarray}

We let $E_{i\a}$ be the matrix with 1 in the intersection of row $i$ and column $\a$ and 0 otherwise. Denote
$g_{i\a,j\be}=\lan E_{i\a},E_{j\be}\ran$ and let
$\big(g^{i\a,j\be}\big)$ be the inverse matrix of
$\big(g_{i\a,j\be}\big)$. Denote by $\n$ the Levi-Civita connection with respect to the canonical matric on $\grs{n}{m}$, and by
\begin{equation*}
\n_{E_{i\a}}E_{j\be}=\G_{i\a,j\be}^{k\g}E_{k\g}.
\end{equation*}
Then from (\ref{metric1}),
\begin{eqnarray}
g_{i\a,j\be}(P)=(1+\la_i^2)^{-1}(1+\la_\a^2)^{-1}\de_{\a\be}\de_{ij}\label{metric3}
\end{eqnarray}
and obviously
\begin{equation}\label{metric4}
g^{i\a,j\be}(P)=(1+\la_i^2)(1+\la_\a^2)\de_{\a\be}\de_{ij}.
\end{equation}
Moreover, a direct calculation from (\ref{metric2}) and (\ref{metric4}) shows
\begin{eqnarray}\label{metric5}\aligned
\G_{i\a,j\be}^{k\g}&=\f{1}{2}g^{k\g,l\de}\big(-E_{l\de}\lan E_{i\a},E_{j\be}\ran+E_{i\a}\lan E_{j\be},E_{l\de}\ran+E_{j\be}\lan E_{l\de},E_{i\a}\ran\big)\\
&=-\la_\a(1+\la_\a^2)^{-1}\de_{\a j}\de_{\be\g}\de_{ik}-\la_\be(1+\la_\be^2)^{-1}\de_{\be i}\de_{\a\g}\de_{jk}.\endaligned
\end{eqnarray}

From (\ref{metric4}), we see that
\begin{equation}
(1+\la_i^2)^{\f{1}{2}}(1+\la_\a^2)^{\f{1}{2}}E_{i\a}\ (1\leq i\leq n,1\leq \a\leq m)
\end{equation}
form an orthonormal basis of $T_P \grs{n}{m}$. Denote its dual basis in $T_P^* \grs{n}{m}$  by
\begin{equation}
\om_{i\a}\ (1\leq i\leq n,1\leq \a\leq m),
\end{equation}
then
\begin{equation}\label{metric6}
g=\sum_{i,\a}\om_{i\a}^2
\end{equation}
at $P$.

\bigskip\bigskip

\Section{Hessian estimates of two  smooths functions on
$\grs{n}{m}$}{Hessian estimates of two  smooths functions on
$\grs{n}{m}$}\label{s3}
\medskip

On $\U$, $w>0$, then we can define
\begin{equation}
v=w^{-1}\qquad \mbox{ on }\U.
\end{equation}
For arbitrary $Q\in \U$ determined by an $n\times m$ matrix $Z$, it is easily seen that
\begin{equation}\label{eq3}
v(Q)=\big[\det(I_n+ZZ^T)\big]^{\f{1}{2}}=\prod_{\a=1}^m \sec\th_\a.
\end{equation}
where $\th_1,\cdots,\th_m$ denotes the Jordan angles between $Q$ and $P_0$.

Now we calculate the Hessian of $v$ at $P$ whose corresponding matrix is $Z_0$. At first, by noting that for any
$n\times n$ orthogonal matrix $U$ and $m\times m$ orthogonal matrix $V$, $Z\mapsto UZV$ induces an isometry of $\U$
which keeps $v$ invariant, we can assume $Z_0=(\la_\a\de_{i\a})$ without loss of generality, where $\la_\a=\tan\th_\a$
and $\th_1,\cdots,\th_m$ denotes the Jordan angles between $P$ and $P_0$.
We also need a Lemma as follows.

\begin{lem}\label{l1}

Let $M$ be a manifold, $A$ be a smooth nonsingular $n\times n$ matrix-valued function on $M$, $X,Y$ be local tangent fields, then
\begin{equation}\label{eq1}
\n_X\log\det A=\text{tr}(\n_X A\cdot A^{-1})
\end{equation}
and
\begin{equation}\label{eq2}
\n_Y\n_X\log\det A=\tr(\n_Y\n_X A\cdot A^{-1})-\tr(\n_X A\cdot A^{-1}\cdot \n_Y A\cdot A^{-1}).
\end{equation}
\end{lem}

\begin{proof}
Assume that $e_1,\cdots,e_n$ is a standard basis in $\R^n$, then
$$\det A\ e_1\w\cdots\w e_n=Ae_1\w\cdots\w Ae_n.$$
Hence
\begin{eqnarray*}\aligned
\n_X \det A\ e_1\w\cdots\w e_n=&\sum_i Ae_1\w\cdots\w Ae_{i-1}\w \n_X Ae_i\w Ae_{i+1}\w\cdots\w Ae_n\\
=&\sum_i Ae_1\w\cdots\w Ae_{i-1}\w (\n_X A\cdot A^{-1})Ae_i\w Ae_{i+1}\w\cdots\w Ae_n\\
=&\tr(\n_X A\cdot A^{-1})Ae_1\w\cdots\w Ae_n\\
=&\tr(\n_X A\cdot A^{-1})\det A\ e_1\w\cdots\w e_n.\endaligned
\end{eqnarray*}
Thereby (\ref{eq1}) immediately follows.

(\ref{eq2}) follows from (\ref{eq1}) and
$$A\cdot \n_Y A^{-1}+\n_Y A\cdot A^{-1}=\n_Y(AA^{-1})=0.$$
\end{proof}

Now we let $M=\U$, $A(Z)=I_n+ZZ^T$, then $\log v=\f{1}{2}\log\det A$. A direct calculation shows
$$\n_X A=XZ^T+ZX^T,\qquad \n_Y\n_X A=XY^T+YX^T.$$
Hence we compute from Lemma \ref{l1} that at $P$
\begin{eqnarray*}\aligned
\n_X \log v&=\f{1}{2}\tr\left((XZ_0^T+Z_0X^T)(I_n+Z_0Z_0^T)^{-1}\right)\\
&=\sum_\a \la_\a(1+\la_\a^2)^{-1}X_{\a\a},
\endaligned\end{eqnarray*}
\begin{eqnarray*}
\n_X\n_Y \log v&=&\f{1}{2}\tr\left((XY^T+YX^T)(I_n+Z_0Z_0^T)^{-1}\right)\\
&&-\f{1}{2}\tr\left((XZ_0^T+Z_0X^T)(I_n+Z_0Z_0^T)^{-1}(YZ_0^T+Z_0Y^T)(I_n+Z_0Z_0^T)^{-1}\right)\\
&=&\sum_{i,\a}(1+\la_i^2)^{-1}X_{i\a}Y_{i\a}\\
&&\qquad-\f{1}{2}\sum_{i,j}(XZ_0^T+Z_0X^T)_{ij}(1+\la_j^2)^{-1}(YZ_0^T+Z_0Y^T)_{ji}(1+\la_i^2)^{-1}\\
&=&\sum_{m+1\leq i\leq n,\a}X_{i\a}Y_{i\a}+\f{1}{2}\sum_{\a,\be}(1+\la_\a^2)^{-1}X_{\a\be}Y_{\a\be}+\f{1}{2}\sum_{\a,\be}(1+\la_\be^2)^{-1}X_{\be\a}Y_{\be\a}\\
&&\qquad-\f{1}{2}\sum_{\a,\be}(\la_\be X_{\a\be}+\la_\a X_{\be\a})(1+\la_\be^2)^{-1}(\la_\a Y_{\be\a}+\la_\be Y_{\a\be})(1+\la_\a^2)^{-1}\\
&&\qquad-\sum_{m+1\leq i\leq n,\a}\la_\a^2(1+\la_\a^2)^{-1}X_{i\a}Y_{i\a}\\
&=&\sum_{m+1\leq i\leq n,\a}(1+\la_\a^2)^{-1}X_{i\a}Y_{i\a}+\sum_{\a,\be}(1+\la_\a^2)^{-1}(1+\la_\be^2)^{-1}X_{\a\be}Y_{\a\be}\\
&&\qquad-\sum_{\a,\be}\la_\a\la_\be(1+\la_\a^2)^{-1}(1+\la_\be^2)^{-1}X_{\a\be}Y_{\be\a}.
\end{eqnarray*} Furthermore,
\begin{eqnarray*}\aligned
\n_X v=&v\n_X \log v=\left(\sum_\a \la_\a(1+\la_\a^2)^{-1}X_{\a\a}\right)v,\\
\n_X\n_Y v=&v(\n_X\n_Y\log v+\n_X\log v\cdot \n_Y \log v)\\
=&\Big(\sum_{m+1\leq i\leq n,\a}(1+\la_\a^2)^{-1}X_{i\a}Y_{i\a}
+\sum_{\a,\be}(1+\la_\a^2)^{-1}(1+\la_\be^2)^{-1}X_{\a\be}Y_{\a\be}\\
&\qquad+\sum_{\a,\be}\la_\a\la_\be(1+\la_\a^2)^{-1}(1+\la_\be^2)^{-1}(X_{\a\a}Y_{\be\be}-X_{\a\be}Y_{\be\a})\Big)v\\
=&\Big(\sum_{i,\be}(1+\la_i^2)^{-1}(1+\la_\be^2)^{-1}X_{i\be}Y_{i\be}\\
&\qquad+\sum_{\a,\be}\la_\a\la_\be(1+\la_\a^2)^{-1}(1+\la_\be^2)^{-1}(X_{\a\a}Y_{\be\be}-X_{\a\be}Y_{\be\a})\Big)v.
\endaligned
\end{eqnarray*}
In particular,
\begin{equation}\label{d1}
\n_{E_{i\a}}v(P)=\la_\a(1+\la_\a^2)^{-1}v\de_{i\a}
\end{equation}
and
\begin{equation}\label{d2}
\n_{E_{i\a}}\n_{E_{j\be}}v(P)=\left\{\begin{array}{cl}
(1+\la_i^2)^{-1}(1+\la_\a^2)^{-1}v & i=j,\a=\be;\\
-\la_\a\la_\be(1+\la_\a^2)^{-1}(1+\la_\be^2)^{-1}v & i=\be,j=\a,\a\neq \be;\\
\la_\a\la_\be(1+\la_\a^2)^{-1}(1+\la_\be^2)^{-1}v & i=\a,j=\be,\a\neq \be;\\
0 & \mbox{otherwise.}
\end{array}\right.
\end{equation}
Then, from (\ref{metric5}), (\ref{d1}) and (\ref{d2}) we obtain
\begin{eqnarray}\label{He}\aligned
\Hess(v)(E_{i\a},E_{j\be})(P)
=&\n_{E_{i\a}}\n_{E_{j\be}}v-(\n_{E_{i\a}}E_{j\be})v\\
=&\n_{E_{i\a}}\n_{E_{j\be}}v-\G_{i\a,j\be}^{k\g}\n_{E_{k\g}}v\\
=&\left\{\begin{array}{cl}
(1+\la_i^2)^{-1}(1+\la_\a^2)^{-1}v & i=j,\a=\be, i\neq \a;\\
(1+2\la_\a^2)(1+\la_\a^2)^{-2}v & i=j=\a=\be;\\
\la_\a\la_\be(1+\la_\a^2)^{-1}(1+\la_\be^2)^{-1}v & i=\be,j=\a,\a\neq \be;\\
\la_\a\la_\be(1+\la_\a^2)^{-1}(1+\la_\be^2)^{-1}v & i=\a,j=\be,\a\neq \be;\\
0 & \mbox{otherwise.}
\end{array}\right.\endaligned
\end{eqnarray}
In other words
\begin{eqnarray}\label{He1}\aligned
\Hess(v)_P&=\sum_{i\neq \a}v\ \om_{i\a}^2+\sum_\a (1+2\la_\a^2)v\
\om_{\a\a}^2
+\sum_{\a\neq\be} \la_\a\la_\be v(\om_{\a\a}\otimes \om_{\be\be}+\om_{\a\be}\otimes\om_{\be\a})\\
&=\sum_{m+1\leq i\leq n,\a}v\
\om_{i\a}^2+\sum_{\a}(1+2\la_\a^2)v\ \om_{\a\a}^2
                                          +\sum_{\a\neq \be}\la_\a\la_\be v\ \om_{\a\a}\otimes\om_{\be\be}\\
&\qquad\qquad+\sum_{\a<\be}\Big[(1+\la_\a\la_\be)v\Big(\f{\sqrt{2}}{2}(\om_{\a\be}
+\om_{\be\a})\Big)^2\\
&\hskip2in+(1-\la_\a\la_\be)v\Big(\f{\sqrt{2}}{2}(\om_{\a\be}-\om_{\be\a})\Big)^2\Big].
\endaligned
\end{eqnarray}
(\ref{He1}) could be simplified further. Please note (\ref{d1}), which also tells us
\begin{equation}\label{d4}
dv=\sum_\a \la_\a v\ \om_{\a\a};
\end{equation}
then
\begin{equation}\label{d3}
dv\otimes dv=\sum_{\a}\la_\a^2 v^2\ \om_{\a\a}^2+\sum_{\a\neq
\be}\la_\a\la_\be v^2\ \om_{\a\a}\otimes\om_{\be\be}.
\end{equation}
Substituting (\ref{d3}) into (\ref{He1}) yields
\begin{eqnarray}\label{He2}\aligned
\Hess(v)_P&=\sum_{m+1\leq i\leq n,\a}v\ \om_{i\a}^2
+\sum_{\a}(1+\la_\a^2)v\ \om_{\a\a}^2 +v^{-1}\
dv\otimes dv\\
&\qquad+\sum_{\a<\be}\Big[(1+\la_\a\la_\be)v\Big(\f{\sqrt{2}}{2}(\om_{\a\be}+\om_{\be\a})\Big)^2\\
&\hskip1in+(1-\la_\a\la_\be)v\Big(\f{\sqrt{2}}{2}(\om_{\a\be}-\om_{\be\a})\Big)^2\Big].
\endaligned
\end{eqnarray}

Note that $\la_\a\geq 0$ and $1-\la_\a\la_\be=1-\tan\th_\a\tan\th_\be=\f{\cos(\th_\a+\th_\be)}{\cos\th_\a\cos\th_\be}$;
which implies that $\Hess(v)_P$ is positive definite if and only if $\th_\a+\th_\be<\f{\pi}{2}$ for
arbitrary $\a\neq \be$, i.e., $P\in B_{JX}(P_0)$.

By (\ref{eq3}), $v=\prod_{\a}(1+\la_\a^2)^{\f{1}{2}}$, then
$$\la_\a\la_\be\leq \big[(1+\la_\a^2)(1+\la_\be^2)\big]^{\f{1}{2}}-1\leq
v-1,$$ the equality holds if and only if $\la_\a=\la_\be$ and
$\la_\g=0$ for each $\g\neq \a,\be$. Hence, we have
$1-\la_\a\la_\be\geq 2-v$. Finally we arrive at an estimate
\begin{equation}\label{es1}
\Hess(v)\geq v(2-v)g+v^{-1}dv\otimes dv.
\end{equation}

Now we introduce another smooth function on $\U$. For any $Q\in \U$,
\begin{equation}
u\mathop{=}\limits^{def.}\sum_\a \tan\th_\a^2.
\end{equation}
where $\th_1,\cdots,\th_m$ denotes the Jordan angles between $Q$ and $P_0$. Denote by $Z$ the coordinate of $Q$,
then it is easily seen that
\begin{equation}\label{eq10}
u(Q)=\tr(ZZ^T).
\end{equation}
We can calculate the Hessian of $u$ at $P\in \U$ whose corresponding matrix is $Z_0$ in the same way. Similar to above, we
can assume $Z_0=\big(\la_a\de_{i\a}\big)$, where $\la_\a=\tan\th_\a$ and
$\th_1,\cdots,\th_m$ are the Jordan angles between $P$ and $P_0$.

Obviously
\begin{eqnarray}\label{d8}\aligned
\n_X u=&\tr(XZ^T)+\tr(ZX^T),\\
\n_X\n_Y u=&\tr(XY^T)+\tr(YX^T).\endaligned
\end{eqnarray}
Then, at $P$
\begin{eqnarray}\label{He3}\aligned
\Hess(u)(E_{i\a},E_{j\be})=&\n_{E_{i\a}}\n_{E_{j\be}}u-(\n_{E_{i\a}}E_{j\be})u\\
=&\n_{E_{i\a}}\n_{E_{j\be}}u-\G_{i\a,j\be}^{k\g}\n_{E_{k\g}}u\\
=&2\de_{ij}\de_{\a\be}+\big(\la_\a(1+\la_\a^2)^{-1}\de_{\a j}\de_{\be \g}\de_{ik}+\la_\be(1+\la_\be^2)^{-1}\de_{\be i}\de_{\a\g}\de_{jk}\big)\\
&\cdot 2\la_\g\de_{k\g}\\
=&2\de_{ij}\de_{\a\be}+2\la_\a\la_\be\big[(1+\la_\a^2)^{-1}+(1+\la_\be^2)^{-1}\big]\de_{\a j}\de_{\be i}\\
=&\left\{\begin{array}{cl}
2 & i=j,\a=\be,i\neq \a;\\
2+4\la_\a^2(1+\la_\a^2)^{-1} & i=j=\a=\be;\\
2\la_\a\la_\be\big[(1+\la_\a^2)^{-1}+(1+\la_\be^2)^{-1}\big] & i=\be,j=\a,\a\neq \be.
\end{array}\right.\endaligned
\end{eqnarray}
In other words
\begin{eqnarray}\label{He4}\aligned
\Hess(u)_P=&\sum_{i\neq \a}2(1+\la_i^2)(1+\la_\a^2)\om_{i\a}^2+\sum_\a (2+6\la_\a^2)(1+\la_\a^2)\om_{\a\a}^2\\
&\qquad+\sum_{\a\neq \be}2\la_\a\la_\be(2+\la_\a^2+\la_\be^2)\om_{\a\be}\otimes \om_{\be\a}\\
=&\sum_{m+1\leq i\leq n,\a}2(1+\la_i^2)(1+\la_\a^2)\om_{i\a}^2+\sum_\a (2+6\la_\a^2)(1+\la_\a^2)\om_{\a\a}^2\\
&+2\big[(1+\la_\a^2)(1+\la_\be^2)+\la_\a\la_\be(2+\la_\a^2+\la_\be^2)\big]
\left[\f{\sqrt{2}}{2}(\om_{\a\be}+\om_{\be\a})\right]^2\\
&+2\big[(1+\la_\a^2)(1+\la_\be^2)-\la_\a\la_\be(2+\la_\a^2+\la_\be^2)\big]
\left[\f{\sqrt{2}}{2}(\om_{\a\be}-\om_{\be\a})\right]^2
\endaligned
\end{eqnarray}

By computing,
\begin{equation}\label{eq4}
2\big[(1+\la_\a^2)(1+\la_\be^2)-\la_\a\la_\be(2+\la_\a^2+\la_\be^2)\big]
=2(1-\la_\a\la_\be)(\la_\a^2+\la_\be^2-\la_\a\la_\be+1).
\end{equation}
It is positive if and only if $1-\la_\a\la_\be=1-\tan\th_\a\tan\th_\be=\f{\cos(\th_\a+\th_\be)}{\cos\th_\a\cos\th_\be}\geq 0$,
i.e., $\th_\a+\th_\be<\f{\pi}{2}$. Hence $\Hess(u)_P$ is positive definite if and only $P\in B_{JX}(P_0)$.

Moreover, the right side of (\ref{eq4}) can be estimated by
\begin{eqnarray*}\aligned
2(1-\la_\a\la_\be)(\la_\a^2+\la_\be^2-\la_\a\la_\be+1)
&\geq 2\big(1-\f{\la_\a^2+\la_\be^2}{2}\big)\big(\f{\la_\a^2+\la_\be^2}{2}+1\big)\\
&=2\big(1-\f{(\la_\a^2+\la_\be^2)^2}{4}\big)\\
&\geq 2(1-\f{1}{4}u^2)=2-\f{1}{2}u^2.\endaligned
\end{eqnarray*}
(Here we used the fact $u=\sum_{\a}\tan^2\th_\a=\sum_\a \la_\a^2$.)
By combining it with (\ref{He4}) and (\ref{metric6}), we arrive that
\begin{equation}\label{es2}
\Hess(u)\geq \left(2-\f{1}{2}u^2\right)g.
\end{equation}

For later applications the estimates (\ref{es1}) and (\ref{es2}) are
not accurate enough. Using the radial compensation technique we
could refine those estimates which are based on the following
lemmas.

\begin{lem}\label{l2}

Let $V$ be a real linear space, $h$ be a nonnegative definite
quadratic form on $V$ and $\om\in V^*$. $V=V_1\oplus V_2$, $h$ is
positive definite on $V_1$, $h(V_1,V_2)=0$ and $\om(V_2)=0$. Denote
by $\om^*$ the unique vector in $V_1$ such that for any $z\in V_1$,
$$\om(z)=h(\om^*,z).$$
Then we have
\begin{equation}\label{eq5}
h\geq \om(\om^*)^{-1}\om\otimes \om.
\end{equation}

\end{lem}

\begin{proof}

For arbitrary $y\in V$, there exist $\la\in \mathbb R$, $z_1\in
V_1$ and $z_2\in V_2$, such that $y=\la \om^*+z_1+z_2$ and
$h(\om^*,z_1)=0$. Then
$$h(y,y)=\la^2 h(\om^*,\om^*)+ h(z_1,z_1)+h(z_2,z_2)\geq \la^2 h(\om^*,\om^*)=\la^2\om(\om^*)$$
and
$$\om(\om^*)^{-1}\om\otimes \om(y,y)=\om(\om^*)^{-1}\om(y)^2=\la^2\om(\om^*).$$
Hence (\ref{eq5}) holds.

\end{proof}

\begin{lem}\label{l3}
Let $\Om$ be a compact and convex subset of $\ir{k}$, such that for every $\si\in \mathscr{S}(k)$ and $x=(x^1,\cdots,x^k)\in \Om$,
\begin{equation}\label{T}
T_\si(x)=(x^{\si(1)},\cdots,x^{\si(k)})\in \Om;
\end{equation}
where $\mathscr{S}(k)$ denotes the permutation group of $\{1,\cdots,k\}$.
If $f:\Om\ra \R$ is a symmetric $C^2$ function, and $(D^2 f)$ is
nonpositive definite everywhere in $\Om$, then there exists $x_0=(x_0^1,\cdots,x_0^k)\in \Om$, such that $x_0^1=x_0^2=\cdots=x_0^k$ and
\begin{equation}\label{eq6}
f(x_0)=\sup_\Om f.
\end{equation}

\end{lem}

\begin{proof}

 By the compactness of $\Om$, there exists $x=(x^1,\cdots,x^k)\in \Om$, such that $f(x)=\sup_\Om f$. Furthermore we have
\begin{equation*}
f\big(T_\si(x)\big)=f(x)=\sup_\Om f\qquad \si\in \mathscr{S}(k)
\end{equation*}
from the fact that $f$ is symmetric. Denote by $C_\si(x)$ the convex closure of $\{T_\si(x):\si\in \mathscr{S}(k)\}$, then $C_\si(x)\subset \Om$
and $f(y)\geq \sup_{\si\in \mathscr{S}(k)}f\big(T_\si(x)\big)=\sup_\Om f$ for arbitrary $y\in C_\si(x)$, since $(D^2 f)\leq 0$; which implies
\begin{equation*}
f\big|_{C_\si(x)}\equiv \sup_\Om f.
\end{equation*}
Denote $x_0^1=\cdots=x_0^k=\f{1}{k}\sum_{i=1}^k x^i$, then
$$x_0=(x_0^1,\cdots,x_0^k)=\f{1}{k}\sum_{s=1}^k (x^s,x^{s+1},\cdots,x^k,x^1,x^2,\cdots,x^{s-1})\in C_\si(x);$$
From which (\ref{eq6}) follows.

\end{proof}

By (\ref{es1}),
\begin{equation}\label{d5}
h\mathop{=}\limits^{def.}\Hess(v)-v(2-v)g-v^{-1}dv\otimes dv
\end{equation}
is nonnegative definite on $T_P \grs{n}{m}$. Denote
\begin{equation}
V_1=\bigoplus_{\a}E_{\a\a},\qquad V_2=\bigoplus_{i\neq \a}E_{i\a} ;
\end{equation}
then $T_P \grs{n}{m}=V_1\oplus V_2$, and (\ref{He2}), (\ref{metric6}), (\ref{d4}) tell us
$$h(V_1,V_2)=0,\ dv(V_2)=0$$
and
\begin{equation}\label{eq7}
h|_{V_1}=\sum_\a (v-1+\la_\a^2)v\ \om_{\a\a}^2.
\end{equation}
is positive definite. Denote by $\td{\n}v$ the unique element in $V_1$ such that for any $X\in V_1$,
$$h(\td{\n}v,X)=dv(X).$$
From (\ref{eq7}) and (\ref{d4}), it is not difficult to obtain
$$\td{\n}v=\sum_\a \f{\la_\a(1+\la_\a^2)}{v-1+\la_\a^2}E_{\a\a}$$
and
\begin{equation}
dv(\td{\n}v)=\sum_\a \f{\la_\a^2}{v-1+\la_\a^2}v.
\end{equation}
Then Lemma \ref{l2} and (\ref{d5}) tell us
\begin{equation}\label{es3}
\Hess(v)\geq v(2-v)g+\Big[1+\big(\sum_\a \f{\la_\a^2}{v-1+\la_\a^2}\big)^{-1}\Big]v^{-1}dv\otimes dv.
\end{equation}
It is necessary to estimate the upper bound of $\sum_\a \f{\la_\a^2}{v-1+\la_\a^2}$.
Denote
\begin{equation}\label{d6}
\nu_\a=\log(1+\la_\a^2),
\end{equation}
then $\la_\a^2=-1+e^{\nu_\a}$; since $v=\prod_\a(1+\la_\a^2)^{\f{1}{2}}$,
\begin{equation*}
\log v=\f{1}{2}\sum_\a \log(1+\la_\a^2)=\f{1}{2}\sum_\a \nu_\a
\end{equation*}
and
\begin{equation*}
\sum_\a \f{\la_\a^2}{v-1+\la_\a^2}=\sum_\a \f{-1+e^{\nu_\a}}{v-2+e^{\nu_\a}}.
\end{equation*}
Now we define
\begin{equation}\label{d7}
\Om=\big\{(\nu_1,\cdots,\nu_m)\in \R^m: \nu_\a\geq 0, \sum_\a \nu_\a=2\log v\big\},
\end{equation}
and $f: \Om\ra \R$ by
\begin{equation*}
(\nu_1,\cdots,\nu_m)\mapsto \sum_\a \f{-1+e^{\nu_\a}}{v-2+e^{\nu_\a}}.
\end{equation*}
Then obviously $\Om$ is compact and convex, $T_\si(\Om)=\Om$ for every $\si\in \mathscr{S}(m)$ (cf. Lemma \ref{l3}), $f$ is a symmetric function and
a direct calculation shows
\begin{equation*}
\f{\p^2 f}{\p \nu_\a\p \nu_\be}=\f{(v-1)e^{\nu_\a}(v-2-e^{\nu_\a})}{(v-2+e^{\nu_\a})^3}\de_{\a\be};
\end{equation*}
i.e.,
\begin{equation*}
(D^2 f)\leq 0\qquad \mbox{when }v\in (1,2].
\end{equation*}
Then from Lemma \ref{l3},
\begin{equation*}
\sup_\Om f=f\big(\f{2\log v}{m},\cdots,\f{2\log v}{m}\big)=\f{m(-1+v^{\f{2}{m}})}{v-2+v^{\f{2}{m}}};
\end{equation*}
which is an upper bound of $\sum_\a \f{\la_\a^2}{v-1+\la_\a^2}$. Substituting it into (\ref{es3}) gives
$$\Hess(v)\geq v(2-v)g+\Big(\f{v-1}{mv(v^{\f{2}{m}}-1)}+\f{m+1}{mv}\Big)dv\otimes dv.$$
In summary, we have the following Proposition.

\begin{pro}

$v$ is a convex function on $B_{JX}(P_0)\subset \U\subset \grs{n}{m}$,  and
\begin{equation}\label{es4}
\Hess(v)\geq v(2-v)g+\Big(\f{v-1}{pv(v^{\f{2}{p}}-1)}+\f{p+1}{pv}\Big)dv\otimes dv
\end{equation}
on $\{P\in \U:v(P)\leq 2\}$,
where $p=min(n,m)$.

\end{pro}

Similarly, we consider
\begin{equation}
\td{h}\mathop{=}\limits^{def.}\Hess(u)-(2-\f{1}{2}u^2)g;
\end{equation}
which is nonnegative definite on $T_P \grs{n}{m}$. The definition of $V_1$ and $V_2$ is similar to above.
It is easily seen from (\ref{He4}) and (\ref{metric6}) that
$$\td{h}(V_1,V_2)=0$$
and
\begin{equation}\label{eq8}
\td{h}|_{V_1}=\sum_\a (8\la_\a^2+6\la_\a^4+\f{1}{2}u^2)\om_{\a\a}^2
\end{equation}
is positive definite.
By (\ref{d8}),
\begin{equation}\label{d9}
du=\sum_{\a}2\la_\a(1+\la_\a^2)\om_{\a\a},
\end{equation}
then
$$du(V_2)=0.$$
Hence Lemma \ref{l2} can be applied for us to obtain
\begin{equation}\label{es5}
\Hess(u)\geq (2-\f{1}{2}u^2)g+\big(du(\td{\n}u)\big)^{-1}du\otimes du.
\end{equation}
where $\td{\n}u$ denotes the unique element in $V_1$ such that for arbitrary $X\in V_1$,
$$\td{h}(\td{\n}u, X)=du(X).$$
From (\ref{eq8}) and (\ref{d9}), we can derive
$$\td{\n}u=\sum_\a \f{2\la_\a(1+\la_\a^2)^2}{8\la_\a^2+6\la_\a^4+\f{1}{2}u^2}E_{\a\a},$$
and hence
\begin{equation}\label{eq9}
du(\td{\n}u)=\sum_\a \f{2\la_\a^2(1+\la_\a^2)^2}{3\la_\a^4+4\la_\a^2+\f{1}{4}u^2}.
\end{equation}
(\ref{es5}) tells us it is necessary for us to estimate the upper bound of the right side of
(\ref{eq9}).

Define $\Om=\{(\nu_1,\cdots,\nu_m)\in \R^m: \sum_\a \nu_\a=u\}$ and $f:\Om\ra \R$
\begin{equation*}
(\nu_1,\cdots,\nu_m)\mapsto \sum_\a \f{2\nu_\a(1+\nu_\a)^2}{3\nu_\a^2+4\nu_\a+C}\qquad \mbox{ where } C=\f{1}{4}u^2.
\end{equation*}
Then it is easy to see that $\sup f$ is an upper bound of $du(\td{\n}u)$, since $u=\sum_\a\tan^2 \th_\a=\sum_\a \la_\a^2.$

Obviously $\Om$ is compact and convex, $T_\si(\Om)=\Om$ for every $\si\in \mathscr{S}(m)$, $f$ is a symmetric function and a direct calculation
shows
\begin{equation*}
\f{\p^2 f}{\p \nu_\a\p \nu_\be}=\f{-4\big[(3C-1)\nu_\a^3+6C\nu_\a^2+(9C-3C^2)\nu_\a+4C-2C^2\big]}{(3\nu_\a^2+4\nu_\a+C)^3}\de_{\a\be}.
\end{equation*}
To show $(D^2 f)\leq 0$ when $u\in (0,2]$, it is sufficient to prove $F:[0,u]\ra \R$
\begin{equation*}
t\mapsto (3C-1)t^3+6Ct^2+(9C-3C^2)t+4C-2C^2
\end{equation*}
is a nonnegative function, where $C=\f{u^2}{4}\in (0,1]$. If $F$ attains its minimum at $t_0\in (0,u)$, then
\begin{eqnarray}
&&0=F'(t_0)=3(3C-1)t_0^2+12Ct_0+9C-3C^2,\\\label{co1}
&&0\leq F''(t_0)=6(3C-1)t_0+12C.\label{co2}
\end{eqnarray}
On the other hand, when $3C-1\geq 0$, we have $F'(t_0)\geq 9C-3C^2>0$, which causes a contradiction; when $3C-1<0$, from (\ref{co2}), $t_0\leq \f{2C}{1-3C}$,
then $F'(t_0)\geq F'(0)=9C-3C^2>0$, which also causes a contradiction. Therefore
\begin{equation*}
\min_{[0,u]}F=\min\big\{F(0),F(u)\big\}.
\end{equation*}
In conjunction with
\begin{eqnarray*}
F(0)&=&4C-2C^2>0\\
F(u)&=&(3C-1)u^3+6Cu^2+(9C-3C^2)u+4C-2C^2\\
&=&\f{9}{16}u^5+\f{11}{8}u^4+\f{5}{4}u^3+u^2>0,
\end{eqnarray*}
$F$ is a nonnegative function. Thereby applying Lemma \ref{l3}  we have
\begin{equation}\label{es6}
du(\td{\n} u)\leq\sup f= f(\f{u}{m},\cdots,\f{u}{m})=\f{2(u+m)^2}{(3+\f{1}{4}m^2)u+4m}.
\end{equation}
Substituting (\ref{es6}) into (\ref{es5}) gives
$$\Hess(u)\geq \left(2-\f{1}{2}u^2\right)g+\f{(3+\f{1}{4}m^2)u+4m}{2(u+m)^2}du\otimes du.$$
We rewrite the conclusion as the following Proposition.
\begin{pro}
$u$ is a convex function on $B_{JX}(P_0)\subset \U\subset \grs{n}{m}$ and
\begin{equation}\label{es7}
\Hess(u)\geq
\left(2-\f{1}{2}u^2\right)g+\f{(3+\f{1}{4}p^2)u+4p}{2(u+p)^2}du\otimes
du
\end{equation}
on $\{P\in \U:u(P)\leq 2\}$,
where $p=min(n,m)$.
\end{pro}
\bigskip\bigskip

\Section{The construction of auxiliary functions}{The construction
of auxiliary functions}
\medskip

Let
\begin{equation}
h_1=v^{-k}(2-v)^k,
\end{equation}
 where $k>0$ to be chosen, then
\begin{eqnarray*}\aligned
h'_1=&-kv^{-k-1}(2-v)^k-kv^{-k}(2-v)^{k-1}\\
=&-2kv^{-k-1}(2-v)^{k-1},\\
h''_1=&2k(k+1)v^{-k-2}(2-v)^{k-1}+2k(k-1)v^{-k-1}(2-v)^{k-2}\\
=&4kv^{-k-2}(2-v)^{k-2}(k+1-v).
\endaligned
\end{eqnarray*}
Here $'$ denotes derivative with respect to $v$. Hence, from
(\ref{es4})
\begin{eqnarray}\label{He5}\aligned
&\Hess(h_1)=-2kv^{-k-1}(2-v)^{k-1}\Hess(v)+4kv^{-k-2}(2-v)^{k-2}(k+1-v)dv\otimes dv\\
&\hskip0.6in\leq-2kv^{-k}(2-v)^k g-\\
&\quad-2kv^{-k-2}(2-v)^{k-2}\Big[\f{(v-1)(2-v)}{p(v^{\f{2}{p}}-1)}+\f{p+1}{p}(2-v)-2(k+1-v)\Big]dv\otimes
dv.
\endaligned
\end{eqnarray}
Please note that $\f{v-1}{v^{\f{2}{p}}-1}$ is an increasing function on $[1,2]$: it is easily seen when $p$ is even, since
$$\f{v-1}{v^{\f{2}{p}}-1}=1+v^{\f{2}{p}}+v^{\f{4}{p}}+\cdots+v^{1-\f{2}{p}};$$
otherwise, when $p$ is odd,
$$\f{v-1}{v^{\f{2}{p}}-1}=\f{v^{1-\f{1}{p}}-1}{v^{\f{2}{p}}-1}+\f{v-v^{1-\f{1}{p}}}{v^{\f{2}{p}}-1}=1+v^{\f{2}{p}}+v^{\f{4}{p}}+\cdots+v^{1-\f{3}{p}}+\f{v^{1-\f{1}{p}}}{v^{\f{1}{p}}+1}$$
it follows from
$$\Big(\f{v^{1-\f{1}{p}}}{v^{\f{1}{p}}+1}\Big)'=\f{1-\f{2}{p}+(1-\f{1}{p})v^{-\f{1}{p}}}{(v^{\f{1}{p}}+1)^2}\geq 0.$$
Hence
$$\f{v-1}{v^{\f{2}{p}}-1}\geq \f{p}{2},$$
and moreover
\begin{eqnarray*}\aligned
\f{(v-1)(2-v)}{p(v^{\f{2}{p}}-1)}&+\f{p+1}{p}(2-v)-2(k+1-v)\\
\geq&\big(\f{1}{2}+\f{p+1}{p}\big)(2-v)-2(k+1-v)\\
=&\big(\f{1}{2}-\f{1}{p}\big)v+\big(3+\f{2}{p}\big)-2(k+1)\\
\geq&\f{3}{2}+\f{1}{p}-2k.\endaligned
\end{eqnarray*}
Now we take
\begin{equation}
k=\f{3}{4}+\f{1}{2p},
\end{equation}
then $\f{(v-1)(2-v)}{p(v^{\f{2}{p}}-1)}+\f{p+1}{p}(2-v)-2(k+1-v)\geq 0$ and then
(\ref{He5}) becomes
\begin{equation}
\Hess(h_1)\leq -2kh_1\ g=-\left(\f{3}{2}+\f{1}{p}\right)h_1\ g.
\end{equation}

Denote
\begin{equation}\label{h2}
h_2=h_1^{-\f{6p}{3p+2}}=v^{\f{3}{2}}(2-v)^{-\f{3}{2}},
\end{equation}
then
\begin{eqnarray}\aligned
\Hess(h_2)=&-\f{6p}{3p+2}h_1^{-\f{6p}{3p+2}-1}\Hess(h_1)+\f{6p}{3p+2}\big(\f{6p}{3p+2}+1\big)h_1^{-\f{6p}{3p+2}-2}dh_1\otimes dh_1\\
\geq& 3h_1^{-\f{6p}{3p+2}}\ g+\big(\f{3}{2}+\f{1}{3p}\big)h_1^{\f{6p}{3p+2}}dh_2\otimes dh_2\\
=& 3h_2\ g+\big(\f{3}{2}+\f{1}{3p}\big)h_2^{-1}dh_2\otimes dh_2.
\endaligned
\end{eqnarray}

Let
\begin{equation}
h_3=(u+\a)^{-1}(2-u),
\end{equation}
where $\a>0$ to be chosen. A direct calculation shows
\begin{eqnarray*}\aligned
h'_3=&-(u+\a)^{-2}(2-u)-(u+\a)^{-1}\\
=&-(2+\a)(u+\a)^{-2},\\
h''_3=&2(2+\a)(u+\a)^{-3}.\endaligned
\end{eqnarray*}
Here $'$ denotes derivative with respect to $u$. Combining with (\ref{es7}), we have
\begin{eqnarray}\label{He6}\aligned
\Hess (h_3)=&-(2+\a)(u+\a)^{-2}\Hess(u)+2(2+\a)(u+\a)^{-3}du\otimes du\\
\leq&-\f{(2+\a)(u+2)}{2(u+\a)}h_3\ g\\
&-(2+\a)(u+\a)^{-3}\Big[\f{(u+\a)\big((3+\f{1}{4}p^2)u+4p\big)}{2(u+p)^2}-2\Big]du\otimes
du.
\endaligned
\end{eqnarray}
Choose
\begin{equation}
\a=p,
\end{equation}
then
\begin{eqnarray*}
\f{(u+\a)\big((3+\f{1}{4}p^2)u+4p\big)}{2(u+p)^2}-2= \f{\big(3+\f{1}{4}p^2\big)u+4p}{2(u+p)}-2\geq 2-2\geq 0,
\end{eqnarray*}
and
$$\f{(2+\a)(u+2)}{2(u+\a)}\geq \f{2+p}{p}=1+\f{2}{p}.$$
Thereby (\ref{He6}) becomes
\begin{equation}
\Hess(h_3)\leq -\big(1+\f{2}{p}\big)h_3\ g.
\end{equation}

Denote
\begin{equation}
h_4=h_3^{-\f{3p}{p+2}}=(u+p)^{\f{3p}{p+2}}(2-u)^{-\f{3p}{p+2}},
\end{equation}
then
\begin{eqnarray}\aligned
\Hess(h_4)=&-\f{3p}{p+2}h_3^{-\f{3p}{p+2}-1}\Hess(h_3)+\f{3p}{p+2}\big(\f{3p}{p+2}+1\big)h_3^{-\f{3p}{p+2}-2}dh_3\otimes dh_3\\
\geq&3h_3^{-\f{3p}{p+2}}\ g+\big(\f{4}{3}+\f{2}{3p}\big)h_3^{\f{3p}{p+2}}dh_4\otimes dh_4\\
=&3h_4\ g+\big(\f{4}{3}+\f{2}{3p}\big)h_4^{-1}dh_4\otimes dh_4.
\endaligned
\end{eqnarray}

Let $M$ be an $n$-dimensional submanifold in $\R^{n+m}$ with $m\ge
2.$ The Gauss map $ \gamma : M \to \grs{n}{m} $ is defined by
$$
 \g (x) = T_x M \in \grs{n}{m}
$$
via the parallel translation in $ \ir{m+n} $ for arbitrary $ x \in M$. The energy density of the Gauss
map (see \cite{x1} Chap.3, \S 3.1) is
$$e(\g)=\f{1}{2}\left<\g_*e_i,\g_*e_i\right>=\f{1}{2}|B|^2.$$
Ruh-Vilms proved  that the mean curvature vector of $M$ is parallel
if and only if its Gauss map is a harmonic map \cite{r-v}.

If the Gauss image of $M$ is contained in $\{P\in \U\subset \grs{n}{m}:v(P)<2\}$, then the composition function $\td{h}_1=h_1\circ \g$ of $h_1$
with the Gauss map $\g$ defines a function on $M$.
Using composition formula, we have
\begin{equation}
\aligned \De \td{h}_1&=\Hess(h_1)(\g_* e_i, \g_* e_i)+d h_1(\tau(\g))\\
&\leq -\big(\f{3}{2}+\f{1}{p}\big) |B|^2\td{h}_1,
\endaligned\label{dh1}
\end{equation}
where $\tau(\g)$ is the tension field of the Gauss map, which is
zero, provided $M$ has parallel mean curvature by the Ruh-Vilms
theorem mentioned above.
Similarly, for $\td{h}_2=h_2\circ \g$ defined on $M$, we have
\begin{equation}
\aligned \De \td{h}_2&=\Hess(h_2)(\g_* e_i, \g_* e_i)+d h_2(\tau(\g))\\
&\geq 3\ \td{h}_2 |B|^2+\big(\f{3}{2}+\f{1}{3p}\big)\td{h}_2^{-1}|\n \td{h}_2|^2.
\endaligned\label{dh2}
\end{equation}

If the Gauss image of $M$ is contained in $\{P\in \U\subset \grs{n}{m}:u(P)<2\}$, we can defined composition function
$\td{h}_3=h_3\circ \g$ and $\td{h}_4=h_4\circ \g$ on $M$. Again using composition formula, we obtain
\begin{equation}\label{dh3}
\De \td{h}_3\leq -\Big(1+\f{2}{p}\Big) |B|^2\td{h}_3
\end{equation}
and
\begin{equation}\label{dh4}
\De \td{h}_4\geq  3\ \td{h}_4
|B|^2+\Big(\f{4}{3}+\f{2}{3p}\Big)\td{h}_4^{-1}|\n \td{h}_4|^2.
\end{equation}

With the aid of $\td{h}_1$ and $\td{h}_3$, we immediately have the following lemma.
\begin{lem}\label{l4}
Let $M$ be an $n$-dimensional minimal submanifold of $\ir{n+m}$
($M$ needs not be complete), if the Gauss image of $M$ is
contained in $\{P\in \U\subset \grs{n}{m}:v(P)<2\}$ (or respectively, $\{P\in \U\subset \grs{n}{m}:u(P)<2\}$),
then we have
\begin{equation}\label{es8}\aligned
\int_M |\n \phi|^2&*1\geq \big(\f{3}{2}+\f{1}{p}\big)\int_M |B|^2\phi^2*1\\
\Big(\mbox{or respectively, }&\int_M |\n \phi|^2*1\geq \big(1+\f{2}{p}\big)\int_M |B|^2\phi^2*1\Big)\endaligned
\end{equation}
for any function $\phi$ with compact support $D\subset M$.
\end{lem}

\begin{rem}
For a stable minimal hypersurface there is the stability inequality,
which is one of main ingredient for Schoen-Simon-Yau's curvature
esimates for stable minimal hypersurfaces. For minimal submanifolds
with the Gauss image restriction we have stronger inequality as
shown in (\ref{es8}). Our proof is similar to \cite{x3} and \cite{x-y}, so we omit the detail of it.
\end{rem}

\bigskip\bigskip

\Section{Curvature estimates}{Curvature estimates }
\medskip

We are now in a position to carry out the curvature estimates of
Schoen-Simon-Yau type.

Let $M$ be an $n$-dimensional minimal submanifold in $\ir{n+m}$.
Assume that the estimate
\begin{equation}\label{lp1}
\int_M |\n \phi|^2*1\geq \la\int_M |B|^2\phi^2*1
\end{equation}
holds for arbitrary function $\phi$ with compact support $D\subset M$, where $\la$ is a positive constant.

Replacing $\phi$ by $|B|^{1+q}\phi$ in (\ref{lp1}) gives
\begin{equation}\label{lp2}\aligned
\int_M |B|^{4+2q}&\phi^2*1\leq \la^{-1}\int_M \big|\n (|B|^{1+q}\phi)\big|^2*1\\
&=\la^{-1}(1+q)^2\int_M |B|^{2q}\big|\n|B|\big|^2\phi^2*1+\la^{-1}\int_M |B|^{2+2q}|\n \phi|^2*1\\
&+2\la^{-1}(1+q)\int_M |B|^{1+2q}\n|B|\cdot \phi\n\phi *1.\endaligned
\end{equation}
Using Bochner technique, the estimate done in \cite{l-l}\cite{c-x}, and the Kato-type inequality derived in \cite{x-y},
we obtain
\begin{equation}\label{es9}
\De |B|^2\geq 2\big(1+\f{2}{mn}\big)\big|\n |B|\big|^2-3|B|^4.
\end{equation}
(For the detail, see \cite{x-y} Section 2.) It is equivalent to
\begin{equation}\label{lp3}
\f{2}{mn}\big|\n|B|\big|^2\leq |B|\De |B|+\f{3}{2}|B|^4.
\end{equation}
Multiplying $|B|^{2q}\phi^2$ with both sides of (\ref{lp3}) and
integrating by parts, we have
\begin{equation}\label{lp4}\aligned
\f{2}{mn}&\int_M |B|^{2q}\big|\n |B|\big|^2\phi^2*1\\
&\qquad\le-(1+2q)\int_M |B|^{2q}\big|\n|B|\big|^2\phi^2*1\\
&\qquad\qquad-2\int_M |B|^{1+2q}\n|B|\cdot \phi\n\phi*1
+\f{3}{2}\int_M |B|^{4+2q}\phi^2 *1. \endaligned
\end{equation}
By multiplying $\f{3}{2}$ with both sides of (\ref{lp2}) and then
adding up both sides of it and (\ref{lp4}), we have
\begin{equation}\label{lp5}\aligned
&\big(\f{2}{mn}+1+2q-\f{3}{2}\la^{-1}(1+q)^2\big)\int_M |B|^{2q}\big|\n |B|\big|^2\phi^2*1\\
&\qquad\leq\f{3}{2}\la^{-1}\int_M |B|^{2+2q}|\n\phi|^2*1
+\big(3\la^{-1}(1+q)-2\big)\int_M |B|^{1+2q}\n|B|\cdot \phi\n\phi
*1.
\endaligned\end{equation}
By using Young's inequality, (\ref{lp5}) becomes
\begin{equation}\label{lp6}
\aligned
&\big(\f{2}{mn}+1+2q-\f{3}{2}\la^{-1}(1+q)^2-\ep\big)\int_M |B|^{2q}\big|\n |B|\big|^2\phi^2*1\\
&\hskip1in\leq C_1(\ep,\la,q)\int_M |B|^{2+2q}|\n\phi|^2*1.\endaligned
\end{equation}
If
\begin{equation}\label{co3}
\la>\f{3}{2}\big(1-\f{2}{mn}\big),
\end{equation}
then
\begin{equation*}
\f{2}{mn}+1+2q-\f{3}{2}\la^{-1}(1+q)^2>0
\end{equation*}
whenever
\begin{equation}
q\in \Big[0,-1+\f{2}{3}\la+\f{1}{3}\sqrt{4\la^2-6\big(1-\f{2}{mn}\big)\la}\ \Big).
\end{equation}
Thus we can choose $\ep$ sufficiently small, such that
\begin{equation}\label{lp7}
\int_M |B|^{2q}\big|\n |B|\big|^2\phi^2*1\leq C_2 \int_M
|B|^{2+2q}|\n\phi|^2*1
\end{equation}
where $C_2$ only depends on $n$, $m$, $\la$ and $q$.

Combining with (\ref{lp2}) and (\ref{lp7}), we can derive
\begin{equation}\label{lp8}
\int_M |B|^{4+2q}\phi^2*1\leq C_3(n,m,\la,q)\int_M |B|^{2+2q}|\n
\phi|^2*1
\end{equation}
by again using Young's inequality.

By replacing $\phi$ by $\phi^{2+q}$ in (\ref{lp8}) and then using H\"{o}lder inequality, we have
\begin{eqnarray}\label{lp9}
\int_M |B|^{4+2q}\phi^{4+2q}*1\leq C\int_M |\n\phi|^{4+2q}*1.
\end{eqnarray}
where $C$ is a constant only depending on $n$, $m$, $\la$ and $q$.

Similarly, replacing $\phi$ by $\phi^{1+q}$ in (\ref{lp8}) and then again using H\"{o}lder inequality yields
\begin{eqnarray}\label{lp10}
\int_M |B|^{4+2q}\phi^{2+2q}*1\leq C'\int_M |B|^2|\n\phi|^{2+2q}*1.
\end{eqnarray}
where $C'$ is a constant only depending on $n$, $m$, $\la$ and $q$.

Let $r$ be a function on $M$ with $|\n r|\le 1$. For any $R\in [0,
R_0]$, where $R_0=\sup_Mr$, suppose
$$M_R=\{x\in M,\quad r\le R\}$$
is compact.

(\ref{lp9}) and Lemma \ref{l4} enable us to prove the following results by taking $\phi\in C_c^\infty (M_R)$ to be the
standard cut-off function such that $\phi\equiv 1$ in $M_{\th R}$
and  $|\n \phi|\leq C(1-\th)^{-1}R^{-1}$.

\begin{thm}\label{t1}
Let $M$ be an $n$-dimensional minimal submanifolds of $\ir{n+m}$.
If  the Gauss image of $M_R$ is contained in $\{P\in \U\subset \grs{n}{m}:v(P)<2\}$, then we have
the estimate
\begin{equation}\label{lp11}
\big\||B|\big\|_{L^s(M_{\th R})}\leq
C(n,m,s)(1-\th)^{-1}R^{-1}\text{Vol}(M_R)^{\f{1}{s}}
\end{equation}
for arbitrary $\th\in (0,1)$ and
$$s\in \left[4,4+\f{4}{3p}+\f{2}{3}\sqrt{\big(3+\f{2}{p}\big)\big(\f{6}{mn}+\f{2}{p}\big)}\ \right).$$

If $p\leq 4$, and the Gauss image of $M_R$ is contained in $\{P\in \U\subset \grs{n}{m}:u(P)<2\}$, then (\ref{lp11}) still holds
for arbitrary $\th\in (0,1)$ and
$$s\in \left[4,2+\f{4}{3}+\f{8}{3p}+\f{2}{3}\sqrt{\big(1+\f{2}{p}\big)\big(\f{12}{mn}+\f{8}{p}-2\big)}\ \right).$$
\end{thm}

We can also fulfil the curvature estimates of Ecker-Huisken type.

Assume that $h$ is a positive function on $M$ satisfying the
following estimate
\begin{equation}\label{co4}
\De h\ge 3h\ g+c_0 h^{-1}dh\otimes dh,
\end{equation}
where
$$c_0>\f{3}{2}-\f{1}{mn}$$
is a positive constant.

We compute from (\ref{co4}) and (\ref{es9}):
\begin{eqnarray*}\aligned
&\De\big(|B|^{2s}h^q\big)\\
\geq& 3(q-s)|B|^{2s+2}h^q+2s(2s-1+\f{2}{mn})|B|^{2s-2}\big|\n|B|\big|^2 h^q
+q(q+c_0-1)|B|^{2s}h^{q-2}|\n h|^2\\
&+4sq|B|^{2s-1}\n|B|\cdot h^{q-1}\n h.\endaligned
\end{eqnarray*}
By Young's inequality, when $2s(2s-1+\f{2}{mn})\cdot
q(q+c_0-1)\geq (2sq)^2$, i.e.,
\begin{equation}
q\geq s\geq \f{1}{2}-\f{1}{mn}+\f{1}{c_0-1}\big(\f{1}{2}-\f{1}{mn}\big)q,
\end{equation}
the inequality
\begin{equation}
\De\big(|B|^{2s}h^q\big)\geq 3(q-s)|B|^{2s+2}h^q
\end{equation}
holds. Especially,
\begin{equation}
\De\big(|B|^{s-1}h^{\f{s}{2}}\big)\geq \f{3}{2}|B|^{s+1}h^{\f{s}{2}}
\end{equation}
whenever
\begin{equation}
s\geq \f{2-\f{2}{mn}}{1-\f{1}{c_0-1}(\f{1}{2}-\f{1}{mn})}.
\end{equation}

Let $\e$ be a smooth function with compact support. Integrating by
parts in conjunction with Young's inequality lead to
\begin{eqnarray}\aligned
\int_M |B|^{2s}h^s \eta^{2s}*1\leq&\f{2}{3}\int_M |B|^{s-1}h^{\f{s}{2}}\eta^{2s}\De\big(|B|^{s-1}h^{\f{s}{2}}\big)*1\\
=&-\f{2}{3}\int_M \Big|\n \big(|B|^{s-1}h^{\f{s}{2}}\big)\Big|^2 \eta^{2s}*1\\
&-\f{2}{3}\int_M |B|^{s-1}h^{\f{s}{2}}\cdot 2s\eta^{2s-1}\n\eta\cdot \n\big(|B|^{s-1}h^{\f{s}{2}}\big)*1\\
\leq&\f{2}{3}s^2\int_M |B|^{2s-2}h^s\eta^{2s-2}|\n \eta|^2*1.\endaligned\label{lp12}
\end{eqnarray}
By H\"{o}lder inequality,
\begin{equation}\label{lp13}
\int_M |B|^{2s-2} h^s \eta^{2s-2}|\n \eta|^2*1
\leq\Big(\int_M
|B|^{2s}h^s\eta^{2s}*1\Big)^{\f{s-1}{s}}\Big(\int_M h^s|\n
\eta|^{2s}*1\Big)^{\f{1}{s}}.
\end{equation}
Substituting (\ref{lp13}) into (\ref{lp12}), we finally arrive at
\begin{equation}
\Big(\int_M |B|^{2s}h^s\eta^{2s}*1\Big)^{\f{1}{s}}\leq
\f{2}{3}s^2\Big(\int_M h^s|\n
\eta|^{2s}*1\Big)^{\f{1}{s}}.\label{lp14}
\end{equation}

Take $\eta\in C_c^\infty (M_R)$ to be the standard cut-off function
such that $\eta\equiv 1$ in $M_{\th R}$ and  $|\n \eta|\leq C
(1-\th)^{-1}R^{-1}$; then  from (\ref{lp14}) we have the following
estimate.

\begin{thm}\label{t2}
Let $M$ be an $n$-dimensional minimal submanifolds of $\ir{n+m}$. If
there exists a positive function $h$ on $M$ satisfying (\ref{co4}), then there exists
$C_1=C_1(n, m,c_0),$ such that
\begin{equation}\label{lp15}
\big\||B|^2 h\big\|_{L^s(M_{\th R})}\leq
C_2(s)(1-\th)^{-2}R^{-2}\big\|h\big\|_{L^s(M_R)}
\end{equation}
whenever $s\geq C_1$ and $\th\in (0, 1)$.
\end{thm}

By (\ref{dh2}) and (\ref{dh4}), if the Gauss image of $M$ is contained in $\{P\in \U\subset \grs{n}{m}:v(P)<2\}$,
or $p\leq 4$ and the Gauss image of $M$ is contained in $\{P\in \U\subset \grs{n}{m}:u(P)<2\}$, there exists
a positive function on $M$, which is $\td{h}_2$ or respectively $\td{h}_4$, satisfying (\ref{co4}). Hence the estimate (\ref{lp15}) holds for both cases.

Furthermore, the mean value inequality for any subharmonic function
on minimal submanifolds in $\ir{m+n}$ (ref. \cite{c-l-y}, \cite{n})
can be applied to yield an estimate of the upper bound of $|B|^2$.
We write the results as the following theorem without detail of proof,
for it is similar to \cite{x-y}. Please note that
$B_R(x)\subset \ir{m+n}$ denotes a ball of radius $R$ centered
at $x\in M$ and its restriction on $M$ is denoted by
$$D_R(x)=B_R(x)\cap M.$$

\begin{thm}\label{t3}
Let $x\in M$, $R>0$ such that the image of $D_R(x)$ under the Gauss
map lies in $\{P\in \U\subset \grs{n}{m}:v(P)<2\}$. Then, there exists $C_1=C_1(n,m)$, such that
\begin{equation}\label{es10}
|B|^{2s} (x)\leq C(n,s)R^{-(n+2s)}(\sup_{D_R(x)}
\td{h}_2)^s\text{Vol}(D_R(x)),
\end{equation}
for arbitrary $s\ge C_1$.

If $p\leq 4$, the image of $D_R(x)$ under the Gauss
map lies in $\{P\in \U\subset \grs{n}{m}:u(P)<2\}$, then there exists $C_2=C_2(n,m)$ such that
\begin{equation}
|B|^{2s} (x)\leq C(n,s)R^{-(n+2s)}(\sup_{D_R(x)}
\td{h}_4)^s\text{Vol}(D_R(x)),
\end{equation} holds for any $s\geq C_2$.
\end{thm}
\bigskip\bigskip

\Section{Bernstein type theorems and related results}{Bernstein
type theorems and related results}
\medskip

If $M$ is a submanifold in $\ir{n+m}$, then the function $w$ defined on $\grs{n}{m}$ (see Section 2)
and the Gauss map $\g$ could be composed, yielding a smooth function on $M$, which is also denoted by $w$.
By studying the properties of $w$-function, we can obtain:

\begin{pro}\label{p1} \cite{x-y}
Let $M$ be a complete submanifold in $\ir{n+m}$. If the $w-$function
is bounded below by a positive constant $w_0$. Then $M$ is an entire
graph with Euclidean volume growth. Precisely,
\begin{equation}
\text{Vol}(D_R(x))\leq \f{1}{w_0}C(n)R^n.
\end{equation}
\end{pro}

Now we let $M$ be
 a complete minimal submanifold in $\R^{n+m}$ whose
Gauss image lies in $\{P\in \U\subset\grs{n}{m}:v(P)<2\}$. Then
$w=v^{-1}>\f{1}{2}$ on $M$ and Proposition \ref{p1} tells us $M$ is
an entire graph. Precisely, the immersion $F:M\to \ir{m+n}$ is
realized by a graph $(x, f(x))$ with $$f:\ir{n}\to\ir{m}.$$ At each
point in $M$ its image $n$-plane $P$ under the Gauss map is spanned
by
$$f_i=\ep_i+\pd{f^\a}{x^i}\ep_\a.$$
Hence the local coordinate of $P$ in $\U$ is
$$Z=\Big(\f{\p f^\a}{\p x^i}\Big).$$
By (\ref{eq3}),
$$v(P)=\left[\text{det}\left(\de_{ij}
     +\sum_\a\pd{f^\a}{x^i}\pd{f^\a}{x^j}\right)\right]^{\f{1}{2}}.$$
Hence
\begin{equation}\label{eq11}
\left[\text{det}\left(\de_{ij}+\sum_\a\pd{f^\a}{x^i}\pd{f^\a}{x^j}\right)\right]^{\f{1}{2}}<2
\end{equation}
at each $x\in \R^n$.
Conversely, if $M=(x,f(x))$ is a minimal graph given by
$f:\R^n\ra \R^m$ which satisfy (\ref{eq11}), then the Gauss image of $M$
lies in $\{P\in \U\subset\grs{n}{m}:v(P)<2\}$.

Let $P\in \U$ such that $u(P)=\sum_\a \tan^2 \th_\a<2$, then $\cos^2\th_\a=(1+\tan^2\th_\a)^{-1}>\f{1}{3}$ and
$$w(P)=\prod_{\a}\cos\th_\a>3^{-\f{p}{2}}.$$
Hence Proposition \ref{p1} could be applied when $M$ is a complete minimal submanifold in $\R^{n+m}$ whose Gauss image
lies in $\{P\in \U\subset\grs{n}{m}:v(P)<2\}$; which is hence a minimal graph given by $f:\R^n\ra \R^m$. Thereby (\ref{eq10}) shows
\begin{equation}
\sum_{i,\a}\left(\f{\p f^\a}{\p x^i}\right)^2<2.
\end{equation}
And vice versa.

Theorem \ref{t1} and Proposition \ref{p1} give us the following Bernstein-type theorem.

\begin{thm}\label{t6}
Let
 $M=(x,f(x))$ be an $n$-dimensional entire minimal graph given by $m$ functions
$f^\a(x^1,\cdots,x^n)$ with $m\geq 2, n\leq 4$. If
$$\De_f=\left[\text{det}\left(\de_{ij}+\sum_\a\pd{f^\a}{x^i}\pd{f^\a}{x^j}\right)\right]^{\f{1}{2}}<2$$
or
$$\La_f=\sum_{i,\a}\left(\f{\p f^\a}{\p x^i}\right)^2<2,$$
then $f^\a$ has to
be affine linear functions representing an affine $n$-plane.
\end{thm}

\begin{proof} If $\De_f<2$, then
the Gauss image of $M$ is contained in $\{P\in \U\subset \grs{n}{m}:v(P)<2\}$. We choose
$$s=4+\f{4}{3p}>4.$$
Fix $x\in M$ and let $r$ be the Euclidean distance function from $x$
and $M_R=D_R(x)$. Hence, letting $R\to +\infty$ in (\ref{lp11})
yields
$$\big\||B|\big\|_{L^s(M)}=0.$$
i.e., $|B|=0$. $M$ has to be an affine linear subspace.

 For the case $\La_f<2$, the proof is similar.
\end{proof}

Theorem \ref{t3} and Proposition \ref{p1} yield Bernstein type
results as follows.
\begin{thm}\label{t4}
Let $M=(x,f(x))$ be an $n$-dimensional entire minimal graph given by
$m$ functions $f^\a(x^1,\cdots,x^n)$ with $m\geq 2$. If
$$\De_f=\left[\text{det}\left(\de_{ij}+\sum_\a\pd{f^\a}{x^i}\pd{f^\a}{x^j}\right)\right]^{\f{1}{2}}<2,$$
and
\begin{equation}
\left(2-\De_f\right)^{-1}=o(R^{\f{4}{3}}),
\end{equation}
where $R^2=|x|^2+|f|^2$.
Then $f^\a$ has to be affine linear functions and hence $M$ has to be
an affine linear subspace.
\end{thm}

\begin{thm}\label{t5}
Let $M=(x,f(x))$ be an $n$-dimensional entire minimal graph given by
$m$ functions $f^\a(x^1,\cdots,x^n)$ with $p=\min\{n,m\}\leq 4$. If
$$\La_f=\sum_{i,\a}\left(\f{\p f^\a}{\p x^i}\right)^2<2,$$
 and
\begin{equation}
\left(2-\La_f\right)^{-1}=o(R^\f{2(p+2)}{3p})
\end{equation}
where $R^2=|x|^2+|f|^2$.
Then $f^\a$ has to be affine linear functions and hence $M$ has to be
an affine linear subspace.
\end{thm}

\begin{proof}

Here we only give the proof of Theorem \ref{t4}, for the proof of Theorem \ref{t5} is similar.

From (\ref{h2}), it is easily seen that
$$h_2\leq C(2-v)^{-\f{3}{2}},$$
where $C$ is a positive constant. Thus, for any point $q\in M$, by
Theorem \ref{t3} and Proposition \ref{p1}, we have
$$|B|^{2s}(q)\le C(n,
s)R^{-2s}\left(2-v\circ\g\right)^{-\f{3}{2}s}$$ Letting $R\to
+\infty$ in the above inequality forces $|B(q)| =0$.

\end{proof}

\begin{rem}

If $n=2$ or $3$, the conclusion of Theorem \ref{t6}-\ref{t5}
could be inferred from the work done by Chern-Osserman
\cite{c-o}, Babosa \cite{ba} and Fischer-Colbrie \cite{f-c}.

\end{rem}

From (\ref{lp10}) it is easy to obtain the following result.

\begin{thm}
Let $M=(x,f(x))$ be an $n$-dimensional entire minimal graph given by
$m$ functions $f^\a(x^1,\cdots,x^n)$. Assume $M$ has finite total
curvature. If $\De_f<2$, or $p\leq 4$ and $\La_f<2$,
 then $M$ has to be an affine linear
subspace.
\end{thm}

There are other applications of the strong stability
inequalities (\ref{es8}), besides its key role in S-S-Y's
estimates. We state following results,
whose detail proof can be found in the previous paper of the first
author \cite{x3}.

\begin{thm}\label{van}
Let $M=(x,f(x))$ be an $n$-dimensional entire minimal graph given by
$m$ functions $f^\a(x^1,\cdots,x^n)$. If $\De_f<2$ or $\La_f<2$,
then any $L^2$-harmonic $1$-form vanishes.
\end{thm}

\begin{thm} Let $M$ be one as in Theorem \ref{van}, $N$ be a manifold
with non-positive sectional curvature. Then any harmonic map $f:M\to
N$ with finite energy has to be constant.
\end{thm}

%%%%%%%%%%%%%%%%%%%%%%%%%%%%%%%%%%%%%%%%%%%%%%%%%%%%%%%%%%%%%%%%%%%%%%%%%%%%%%
\bibliographystyle{amsplain}

\begin{thebibliography}{10}

\bibitem{ba} J.L.M.Babosa: An extrinsic rigidity theorem for minimal
immersion from $S^2$ into $S^n$. J. differntial Geometry {\bf 14(3)}
(1980), 355-368.

\bibitem{b} S.~Bernstein:
Sur un th\'{e}or\`{e}me de g\'{e}om\'{e}trie et ses application aux
\'{e}quations aux d\'{e}riv\'{e}s partielles du type elliptique.
Comm. de la Soc Math. de Kharkov (2\'{e} s\'{e}r.) {\bf 15}
(1915-1917), 38-45.

\bibitem{b-d-g} E.~Bombieri, E. De Giorgi and E. Guusti:
Minimal cones and Bernstein problem. Invent. Math. {\bf
7}(1969), 243-268.

\bibitem{d} Yuxin Dong: On graphic submanifolds with parallel mean
curvature in Euclidean space. Preprint.

\bibitem{ch}S. S. Chern: On the curvature of a piece of
hypersurfaces in Euclidean space, Abh. Math. Sem. Univ. Hamberg
{\bf 29} (1965), 77-91.

\bibitem{c-o} S. S. Chern and R. Osserman: Complete minimal surfaces
in Euclidean $n-$space. J. d'Anal. Math. {\bf 19}(1967), 15-34.
\bibitem{c-x} Qing Chen and Senlin Xu:
Rigidity of compact minimal submanifolds in a unit sphere. Geom.
Dedicata {\bf 45 (1)}(1993), 83-88.

\bibitem{c-l-y}S.~Y.~Cheng, P.~Li  and S.~T.~Yau:
Heat equations on minimal submanifols and their applications. Amer.
J. Math. {\bf  103}(1981), 1021-1063.

\bibitem{e-h} K.~Ecker  and  G.~Huisken:
A Bernstein result for minimal graphs of controlled growth. J. Diff.
Geom. {\bf 31}(1990), 397-400.

\bibitem{f-c} D.~Fischer-Colbrie:
Some rigidity theorems for minimal submanifolds of the sphere. Acta math.
{\bf 145}(1980), 29-46.

\bibitem{h-j-w}S.~Hildebrandt, J.~Jost,~J and K.~O.~Widman:
Harmonic mappings and minimal submanifolds.  Invent. math. {\bf 62}
(1980), 269-298.

\bibitem{j-x} J.~Jost and  Y.~L.~Xin:
Bernstein type theorems for higher codimension. Calculus. Var. PDE
{\bf 9} (1999), 277-296.

\bibitem{l-o} H. B.~Lawson and R.~Osserman:
Non-existence, non-uniqueness and irregularity of solutions to the
minimal surface system. Acta math. {\bf 139}(1977), 1-17.

\bibitem{l-l}An-Min Li and Jimin Li:
An intrinsic rigidity theorem for minimal submanifolds in a sphere.
Arch. Math.{\bf 58} (1992), 582-594.

\bibitem{mo} J.~Moser:
On Harnack's theorem for elliptic differential equations.
 Comm. Pure Appl. Math. {\bf 14} (1961), 577-591.

\bibitem{n} Lei Ni:  Gap theorems for minimal submanifolds in
$\ir{n+1}$. Comm. Analy.  Geom. {\bf 9 (3)}(2001), 641-656.


\bibitem{r-v} E. A. Ruh and J. Vilms: The tension field of Gauss map.
Trans. Amer. Math. {\bf 149}(1970), 569-573.


\bibitem{s-s-y} R.~ Schoen, L.~ Simon and S.~ T.~ Yau:
Curvature estimates for minimal hypersurfaces. Acta Math. {\bf 134}
(1975), 275-288.

\bibitem{si} J. Simons:
Minimal varieties in Riemannian manifolds. Ann. Math. {\bf 88}
(1968), 62-105.

\bibitem{s-w-x}K.~Smoczyk, Guofang Wang and Y.~L.~Xin:
Bernstein type theorems with flat normal bundle. Calc. Var. and PDE.
{\bf 26(1)}(2006), 57-67.

\bibitem{x1} Y. L. Xin: Geometry of harmonic maps.
Birkh\"auser PNLDE 23, (1996).


\bibitem{x2} Yuanlong Xin:
Minimal submanifolds and related topics. World Scientific Publ.,
(2003).

\bibitem{x3} Y.~L.~Xin:
Bernstein type theorems without graphic condition. Asia J. Math.
{\bf 9(1)}(2005), 31-44.

\bibitem{x-y} Y.~L.~Xin and Ling Yang:
Curvature estimates for minimal submanifolds of higher codimension.
arXiv:0709.3686.





\end{thebibliography}

\end{document}